\documentclass[a4paper, 11pt]{amsart}
\usepackage{amsmath, amssymb, amsthm, mathrsfs}

\theoremstyle{plain} 
 \newtheorem{thm}{Theorem}[section]
 \newtheorem{lem}[thm]{Lemma}
 
 \newtheorem{prop}[thm]{Proposition}
 
\theoremstyle{definition}

\theoremstyle{remark}
  \newtheorem{rem}[thm]{Remark}

\newcommand{\R}{\mathbb{R}}
\newcommand{\cal}{\mathcal}
\newcommand{\N}{\mathbb{N}}
\newcommand{\Z}{\mathbb{Z}}
\newcommand{\calr}{\mathcal{R}}
\newcommand{\T}{\mathbb{T}}

\renewcommand{\c}{\curvearrowright}

\begin{document}

\title{Inner amenable groups having no stable action}
\author{Yoshikata Kida}
\address{Department of Mathematics, Kyoto University, 606-8502 Kyoto, Japan}
\email{kida@math.kyoto-u.ac.jp}
\date{November 29, 2012}
\subjclass[2010]{20E06, 20E08, 37A20}
\keywords{Inner amenability, orbit equivalence, stability}

\begin{abstract}
We construct inner amenable and ICC groups having no ergodic, free, probability-measure-preserving and stable action.
This solves a problem posed by Jones-Schmidt in 1987.
\end{abstract}

\maketitle


\section{Introduction}

Unless otherwise stated, we assume that a discrete group is countable.
We say that a discrete group $\Gamma$ is {\it inner amenable} if there exists a sequence of non-negative $\ell^1$-functions on $\Gamma$ with unit norm, $(f_n)_{n\in \N}$, such that for any $\gamma \in \Gamma$, we have
\[\lim_{n\to \infty}f_n(\gamma)=0\quad \textrm{and}\quad \lim_{n\to \infty}\sum_{\delta \in \Gamma}|f_n(\gamma^{-1}\delta \gamma)-f_n(\delta)|=0.\]
A discrete group $\Gamma$ is called {\it ICC} if the conjugacy class of any non-neutral element of $\Gamma$ in $\Gamma$ has infinitely many elements, where ``ICC" stands for ``infinite conjugacy classes".
Inner amenability was introduced by Effros \cite{effros} to determine ICC groups whose von Neumann algebra have property Gamma.
Effros proved that any discrete ICC group whose von Neumann algebra has property Gamma is inner amenable, and posed the question whether the converse holds.
Recently, Vaes \cite{vaes} constructed a counterexample to this question.
We refer to \cite{bh} for a survey on inner amenability of groups.

Jones-Schmidt \cite{js} observed an analogue in the setting of orbit equivalence for group actions.
We mean by a {\it p.m.p}.\ action of a discrete group $\Gamma$ a measure-preserving action of $\Gamma$ on a standard probability space, where ``p.m.p." stands for ``probability-measure-preserving''.
We say that a p.m.p.\ action of a discrete group is {\it stable} if the associated equivalence relation is isomorphic to its direct product with the ergodic hyperfinite equivalence relation of type ${\rm II}_1$.
Any discrete group having an ergodic, free, p.m.p.\ and stable action is inner amenable (\cite[Proposition 4.1]{js}).
In Problem 4.2 of \cite{js}, Jones-Schmidt asked the converse:
Does any inner amenable group have an ergodic, free, p.m.p.\ and stable action?
In \cite{kida-stab}, we constructed such a stable action of Baumslag-Solitar groups and of the group of Vaes mentioned above.
The aim of this paper is to give a counterexample to the Jones-Schmidt question by proving the following:

\begin{thm}\label{thm-main}
Let $E$ be a discrete group with property (T) and having a central element $a$ of infinite order.
Let $p$ and $q$ be non-zero integers with $|p|\neq |q|$.
Then the group $\Gamma$ defined by the presentation
\[\Gamma =\langle \, E,\ t\mid ta^pt^{-1}=a^q\, \rangle\]
is inner amenable and ICC, and has no ergodic, free, p.m.p.\ and stable action.
\end{thm}

\begin{rem}\label{rem-ex}
Examples of groups with property (T) and having a central element of infinite order are found in \cite[Example 1.7.13]{bhv}.
One of them is the following:
Let $n$ be an integer with $n\geq 2$.
The fundamental group of the Lie group $G=Sp_{2n}(\R)$ is isomorphic to $\Z$.
Let $\widetilde{G}$ denote the universal covering of $G$.
We set $L=Sp_{2n}(\Z)$, which is a lattice in $G$.
The inverse image $\widetilde{L}$ of $L$ in $\widetilde{G}$ then has property (T) and has a central subgroup isomorphic to $\Z$.
\end{rem}

\begin{rem}
A discrete group with property (T) and with infinite center is a {\it non-ICC} counterexample to the Jones-Schmidt question because
\begin{itemize}
\item any discrete group with infinite center is inner amenable; and
\item no ergodic, free and p.m.p.\ action of a discrete group with property (T) is stable (\cite[Proposition 2.10 and Remark 2.11]{sch}).
\end{itemize}
On the other hand, for a discrete ICC group, it is impossible to have inner amenability and property (T) simultaneously (\cite[Corollaire 3 (i)]{bh}).
In this respect, it is meaningful to ask the Jones-Schmidt question for ICC groups, a counterexample to which is the group $\Gamma$ in Theorem \ref{thm-main}.
\end{rem}

The following theorem is obtained by Narutaka Ozawa, and implies that the von Neumann algebra of the group $\Gamma$ in Theorem \ref{thm-main} has property Gamma.

\begin{thm}[N.\ Ozawa]\label{thm-gamma}
Let $E$ be a discrete group having a central element $a$ of infinite order.
Let $p$ and $q$ be non-zero integers with $|p|\neq |q|$.
We define $\Gamma$ as the group with the presentation
\[\Gamma =\langle \, E,\ t\mid ta^pt^{-1}=a^q\, \rangle.\]
Then $\Gamma$ is ICC, and the von Neumann algebra of $\Gamma$ has property Gamma.
\end{thm}

For the reader's convenience, we give a proof of Theorem \ref{thm-gamma} in Section \ref{sec-gamma}.
We thank Narutaka Ozawa for allowing us to include it in this paper.

In Section \ref{sec-ter}, we recall terminology on discrete measured equivalence relations.
In Section \ref{sec-proof}, we prove Theorem \ref{thm-main}.
In Section \ref{sec-gamma}, we prove Theorem \ref{thm-gamma}.


\section{Terminology}\label{sec-ter}

Unless otherwise mentioned, all relations among measurable sets and maps that appear in this paper are understood to hold up to sets of measure zero.
Let $(X, \mu)$ be a standard probability space.
Let $\calr$ be a discrete measured equivalence relation on $(X, \mu)$ of type ${\rm II}_1$.
For $x\in X$, we set $\calr_x=\{ \, (y, x)\in \calr \mid y\in X\, \}$.
We have the measure $\tilde{\mu}$ on $\calr$ defined by
\[\tilde{\mu}(D)=\int_X|D\cap \calr_x|\, d\mu(x)\]
for a Borel subset $D$ of $\calr$, where for a set $S$, we denote by $|S|$ its cardinality.
For a Borel subset $A$ of $X$ with positive measure, we define $\calr|_A$ as the restriction $\calr \cap (A\times A)$, which is a discrete measured equivalence relation on $(A, \mu|_A)$.

Suppose that we have a discrete group $\Gamma$ acting on a standard Borel space $S$.
Let $\rho \colon \calr \to \Gamma$ be a Borel cocycle, i.e., a Borel homomorphism in the sense of discrete measured groupoids.
Let $A$ be a Borel subset of $X$ with positive measure.
We say that a Borel map $\varphi \colon A\to S$ is {\it $(\calr, \rho)$-invariant} if for $\tilde{\mu}$-a.e.\ $(y, x)\in \calr|_A$, the equation $\rho(y, x)\varphi(x)=\varphi(y)$ holds.


\section{Proof of Theorem \ref{thm-main}}\label{sec-proof}

Let $E$ be a discrete group with property (T) and having a central element $a$ of infinite order.
Let $p$ and $q$ be non-zero integers with $|p|\neq |q|$.
We define $\Gamma$ as
\[\Gamma =\langle \, E,\ t\mid ta^pt^{-1}=a^q\, \rangle.\]
This is the HNN extension of $E$ relative to the isomorphism $\phi$ from $\langle a^p\rangle$ onto $\langle a^q\rangle$ sending $a^p$ to $a^q$, where for each $b\in E$, we denote by $\langle b\rangle$ the subgroup generated by $b$.
Let $r>0$ be the greatest common divisor of $p$ and $q$, and set $p_0=p/r$ and $q_0=q/r$.

\medskip

\noindent {\bf Having ICC and inner amenability.}
For any positive integer $j$, the isomorphism $\phi^j$ from $\langle a^{rp_0^j}\rangle$ onto $\langle a^{rq_0^j}\rangle$ has no non-neutral fixed element because $|p|\neq |q|$.
We can thus apply \cite[Th\'eor\`eme 0.1]{stalder}, and see that $\Gamma$ is ICC.
For any positive integer $n$, we have $\phi^j(a^{rp_0^n})=a^{rp_0^{n-j}q_0^j}$ for any $j=1,\ldots, n$.
We can thus apply \cite[Proposition 0.2]{stalder}, and see that $\Gamma$ is inner amenable.

\medskip

\noindent {\bf The Bass-Serre tree.}
Let $T$ be the Bass-Serre tree associated with the decomposition of $\Gamma$ into an HNN extension.
The set of vertices of $T$, denoted by $V(T)$, is defined as $\Gamma /E$.
The set of edges of $T$, denoted by $E(T)$, is defined as $\Gamma /\langle a^q\rangle$, and for each $\gamma \in \Gamma$, the edge corresponding to the coset $\gamma \langle a^q\rangle$ joins the two vertices corresponding to the cosets $\gamma E$ and $\gamma tE$.
The group $\Gamma$ acts on $V(T)$ and $E(T)$ by left multiplication, and acts on $T$ by simplicial automorphisms.
We refer to \cite{serre} for the Bass-Serre theory.

We define a path metric on $T$ so that each edge of $T$ is isometric to the unit interval with the Euclidean metric.
Let $\partial T$ denote the ideal boundary of $T$ as a hyperbolic metric space. 
We define $\Delta T$ as the disjoint union $V(T)\cup \partial T$. 
For any two distinct points $x$, $y$ in $\Delta T$, there exists a unique geodesic path in $T$ from $x$ to $y$, denoted by $l_x^y$.
For $x\in \Delta T$ and a subset $A$ of $V(T)$, we define the set
\[M(x, A)=\{ \, y\in \Delta T\mid l_x^y\cap (A\setminus \{ x\})=\emptyset \, \}.\]
The collection of the sets $M(x, A)$ for all $x\in \Delta T$ and all finite subsets $A$ of $V(T)$ defines an open basis for a topology on $\Delta T$.
With respect to this topology, $\Delta T$ is compact and Hausdorff, and any simplicial automorphism of $T$ uniquely extends to a homeomorphism from $\Delta T$ onto itself.
We refer to \cite[Section 8]{bow} for details on this topology.

We define $\partial_2 T$ as the quotient space of $\partial T\times \partial T$ by the action of the symmetric group of two letters that exchanges the coordinates.
We have the action of $\Gamma$ on $\partial_2 T$ induced by the diagonal action of $\Gamma$ on $\partial T\times \partial T$.
Let $M(\Delta T)$ denote the space of probability measures on $\Delta T$.
We have the natural $\Gamma$-equivariant embedding of $\partial_2 T$ into $M(\Delta T)$.

\begin{prop}\label{prop-nor}
Let $\Gamma \c (X, \mu)$ be a free and p.m.p.\ action, and denote by $\calr$ the associated equivalence relation, i.e., $\calr =\{ \, (\gamma x, x)\in X\times X\mid \gamma \in \Gamma,\ x\in X\, \}$.
We define a Borel cocycle $\rho \colon \calr \to \Gamma$ by $\rho(\gamma x, x)=\gamma$ for $\gamma \in \Gamma$ and $x\in X$.
If $\cal{S}$ is an amenable normal subrelation of $\calr$, then there exists an $(\cal{S}, \rho)$-invariant Borel map $\varphi \colon X\to V(T)$.
\end{prop}

We refer to \cite[Definition 3.15]{kida-bs} for the definition of normal subrelations, which were originally introduced by Feldman-Sutherland-Zimmer \cite{fsz}.
In \cite{kida-bs}, more generally, normal subgroupoids of a discrete measured groupoid are dealt with.

\begin{proof}[Proof of Proposition \ref{prop-nor}]
This proof is based on Adams' argument in \cite[Sections 3--6]{adams}, and is similar to the proof of \cite[Theorem 5.1]{kida-bs}.
We thus give only a sketch of the proof.
For a Borel subset $Y$ of $X$ with positive measure, we define $I_Y$ as the set of all $(\cal{S}, \rho)$-invariant Borel maps from $Y$ into $V(T)$.
To prove the proposition, it suffices to deduce a contradiction under the assumption that there exists a Borel subset $A$ of $X$ with positive measure such that for any Borel subset $B$ of $A$ with positive measure, $I_B$ is empty.

Amenability of $\cal{S}$ implies that there exists an $(\cal{S}, \rho)$-invariant Borel map $\varphi$ from $X$ into $M(\Delta T)$.
By our assumption, for a.e.\ $x\in A$, the measure $\varphi(x)$ is supported on $\partial T$.
Following the proof of \cite[Lemma 5.3]{kida-bs}, we can show that for a.e.\ $x\in A$, the measure $\varphi(x)$ is supported on at most two points of $\partial T$.
It follows that $\varphi$ induces an $(\cal{S}, \rho)$-invariant Borel map from $A$ into $\partial_2 T$.

We define $J$ as the set of all $(\cal{S}, \rho)$-invariant Borel maps from $A$ into $\partial_2 T$, which is non-empty by the argument in the last paragraph.
For each $\varphi \in J$, we set
\[S_{\varphi}=\{ \, x\in A\mid |{\rm supp}(\varphi(x))|=2\, \},\]
where ${\rm supp}(\nu)$ denotes the support of a measure $\nu$.
There exists an element $\varphi_0$ of $J$ with $\mu(S_{\varphi_0})=\sup_{\varphi \in J}\mu(S_{\varphi})$.
For any $\varphi \in J$, we have the inclusion $\varphi(x)\subset \varphi_0(x)$ for a.e.\ $x\in A$, where each element of $\partial_2 T$ is naturally regarded as a subset of $\partial T$ consisting of one or two points.

Using this maximality of $\varphi_0$ and the assumption that $\cal{S}$ is normal in $\calr$, we can show that $\varphi_0$ is $(\calr, \rho)$-invariant.
We extend $\varphi_0$ to the $(\calr, \rho)$-invariant Borel map from the saturation $\calr A=\Gamma A$ into $\partial_2 T$ uniquely.
This map induces a $\Gamma$-invariant probability measure on $\partial_2 T$, denoted by $\nu_0$.
On the other hand, by \cite[Corollary 3.4]{kida-exama}, the action of $\Gamma$ on the measure space $(\partial_2 T, \nu_0)$ is amenable.
By \cite[Proposition 4.3.3]{zim-book}, $\Gamma$ is amenable.
This is a contradiction.
\end{proof}

\begin{lem}
The group $\Gamma$ has no ergodic, free, p.m.p.\ and stable action.
\end{lem}

\begin{proof}
Suppose that there exists an ergodic, free, p.m.p.\ and stable action $\Gamma \c (X, \mu)$.
Let $\calr$ denote the associated equivalence relation, and define a Borel cocycle $\rho \colon \calr \to \Gamma$ as in Proposition \ref{prop-nor}.
Since the action $\Gamma \c (X, \mu)$ is stable, we have two discrete measured equivalence relations on standard probability spaces of type ${\rm II}_1$, $\cal{S}$ on $(Y, \nu)$ and $\cal{T}$ on $(Z, \xi)$, such that $\cal{S}$ is hyperfinite, and $\calr$ is isomorphic to the direct product
\[\cal{S}\times \cal{T} =\{ \, ((y_1, z_1), (y_2, z_2))\in (Y\times Z)^2\mid (y_1, y_2)\in \cal{S},\ (z_1, z_2)\in \cal{T}\, \}\]
on $(Y\times Z, \nu \times \xi)$.
Let $f\colon \calr \to \cal{S}\times \cal{T}$ be an isomorphism.

We define $\cal{I}$ as the trivial equivalence relation on $(Z, \xi)$.
Since $\cal{S}\times \cal{I}$ is amenable and normal in $\cal{S}\times \cal{T}$, there exists an $(f^{-1}(\cal{S}\times \cal{I}), \rho)$-invariant Borel map $\varphi \colon X\to V(T)$ by Proposition \ref{prop-nor}.
Pick a Borel subset $A$ of $X$ with positive measure on which $\varphi$ is constant with the value $v\in V(T)$.
Let $\Gamma_v$ be the stabilizer of $v$ in $\Gamma$.
Let $\calr_v$ be the subrelation of $\calr$ generated by $\Gamma_v$.
The inclusion $f^{-1}(\cal{S}\times \cal{I})|_A<\calr_v$ then holds.

On the other hand, since $\cal{S}$ is ergodic, hyperfinite and of type ${\rm II}_1$, there exists an ergodic, free and p.m.p.\ action $\Z \c (Y, \nu)$ such that the associated equivalence relation is equal to $\cal{S}$.
We have the Borel cocycle from $\cal{S}$ into $\Z$ sending $(ny, y)$ to $n$ for $n\in \Z$ and $y\in Y$.
Define a Borel cocycle $\sigma \colon \cal{S}\times \cal{T}\to \Z$ as the composition of this cocycle with the projection from $\cal{S}\times \cal{T}$ onto $\cal{S}$.
The group $\Gamma_v$ is a conjugate of $E$ in $\Gamma$, and thus has property (T).
By \cite[Theorem 9.1.1]{zim-book}, there exists a Borel map $\psi \colon X\to \Z$ such that for any $\gamma \in \Gamma_v$ and a.e.\ $x\in X$, the equation $\sigma \circ f(\gamma x, x)=\psi(\gamma x)-\psi(x)$ holds.
Pick a Borel subset $B$ of $A$ with positive measure on which $\psi$ is constant.
The equation implies that $f^{-1}(\cal{S}\times \cal{I})|_B$ is trivial.
This is a contradiction.
\end{proof}

The proof of Theorem \ref{thm-main} is completed.


\section{Proof of Theorem \ref{thm-gamma}}\label{sec-gamma}

Let $E$ be a discrete group having a central element $a$ of infinite order.
Let $p$ and $q$ be non-zero integers with $|p|\neq |q|$.
We define $\Gamma$ as the group with the presentation
\[\Gamma =\langle \, E,\ t\mid ta^pt^{-1}=a^q\, \rangle.\]
Applying \cite[Th\'eor\`eme 0.1]{stalder} as in the beginning of Section \ref{sec-proof}, we see that $\Gamma$ is ICC.

Let $L\Gamma$ be the von Neumann algebra of $\Gamma$ with trace $\tau$, generated by the unitary elements $u_\gamma$ with $\gamma \in \Gamma$.
Let $\Vert \cdot \Vert_2$ denote the $L^2$-norm on $L\Gamma$ defined by $\Vert x\Vert_2=\tau(x^*x)^{1/2}$ for $x\in L\Gamma$.
We show that $L\Gamma$ has {\it property Gamma}, that is, there exists a sequence of unitary elements of $L\Gamma$, $(x_n)_{n\in \N}$, such that for any $n\in \N$, $\tau(x_n)=0$, and for any $y\in L\Gamma$, $\Vert x_ny-yx_n\Vert_2\to 0$ as $n\to \infty$.

\medskip

\noindent {\bf Covering maps of the torus.}
We set $\T =\{ \, z\in \mathbb{C}\mid |z|=1\, \}$, and denote by $\mu$ the normalized Haar measure on $\T$.
We identify $\T$ with the dual of the group $\langle a \rangle$.
The algebra $L^{\infty}(\T)$ is then identified with the subalgebra $L\langle a\rangle$ of $L\Gamma$ so that the identity function on $\T$ corresponds to $u_a$.

For $k\in \Z$, we define a map $T_k\colon \T \to \T$ by $T_k(z)=z^k$ for $z\in \T$.
Note that for any non-zero integer $k$ and any measurable subset $A$ of $\T$, we have $\mu(T_k^{-1}(A))=\mu(A)$.
The map $T_k^{-1}$ from the set of measurable subsets of $\T$ into itself therefore induces an isometry from $L^2(\T, \mu)$ into itself.
For each $f\in L^2(\T, \mu)$, the image of $f$ under this isometry is denoted by $T_k^{-1}f$, which is equal to the function $f\circ T_k$.

\medskip

For $m\in \Z$, let $u(m)$ denote the element of $L^{\infty}(\T)$ corresponding to the unitary element $u_{a^m}$ of $L\langle a\rangle$.
For $n\in \N$, set $u_n=\sum_{i, j=1}^nu(p^iq^j)$, and set $v_n=u_n/\Vert u_n\Vert_2$.
For any $n\in \N$, we then have $\Vert v_n\Vert_2=1$ and $\tau(v_n)=0$.
For any $k\in \Z \setminus \{ 0\}$ and any $m\in \Z$, the equation $T_k^{-1}(u(m))=u(km)$ holds.
It implies that for any $k=p, q$, we have $\Vert T_k^{-1}v_n-v_n\Vert_2\to 0$ as $n\to 0$.

\begin{lem}\label{lem-a}
There exists a sequence of measurable subsets of $\T$, $(B_n)_{n\in \N}$, such that for any $n\in \N$, $\mu(B_n)=1/2$, and for any $k=p, q$, we have
\[\lim_{n\to \infty}\mu(T_k^{-1}(B_n)\bigtriangleup B_n)=0.\]
\end{lem}

To prove this lemma, we need the following:

\begin{lem}\label{lem-mix}
Let $k$ be an integer with $k\geq 2$.
Then for any measurable subsets $A$, $B$ of $\T$, we have
\[\lim_{l\to \infty}\mu(T_k^{-l}(A)\cap B)=\mu(A)\mu(B).\]
\end{lem}

\begin{proof}
Let $\Z_-$ be the set of negative integers.
Define $X$ as the standard probability space $\prod_{n\in \Z_-}\{ 0, 1,\ldots, k-1\}$ with the product measure of the uniformly distributed probability measure on the set $\{ 0, 1,\ldots, k-1\}$.
The map from $X$ into $\T$ sending each element $(x_n)_{n\in \Z_-}$ of $X$ to the number $\exp(2\pi i \sum_{n\in \Z_-}x_nk^n)$ in $\T$ is an isomorphism between $X$ and $(\T, \mu)$.
Under this isomorphism, for any measurable subset $A$ of $X$, the inverse image of $A$ under $T_k$ is equal to the shifted set
\[\{ \, (x_n)_{n\in \Z_-}\in X \mid (x_{n-1})_{n\in \Z_-} \in A \, \}.\]
The lemma follows from this equality.
\end{proof}

\begin{proof}[Proof of Lemma \ref{lem-a}]
Using the convergence $\Vert T_k^{-1}v_n-v_n\Vert_2\to 0$ as $n\to \infty$ for any $k=p, q$, we follow the proof of ``(1) $\Rightarrow$ (3)" in \cite[Proposition 2.3]{sch-ame}.
We can then find a sequence of measurable subsets of $\T$, $(A_n)_{n\in \N }$, such that either
\begin{enumerate}
\item[(I)] for any $k=p, q$, we have $\mu(T_k^{-1}(A_n)\bigtriangleup A_n)\to 0$ as $n\to \infty$, and there exist numbers $c_1$, $c_2$ with $0<c_1\leq \mu(A_n)\leq c_2<1$ for any $n\in \N$; or
\item[(II)] $\mu(A_n)>0$ for any $n\in \N$, $\mu(A_n)\to 0$ as $n\to \infty$, and for any $k=p, q$, we have $\mu(T_k^{-1}(A_n)\bigtriangleup A_n)/\mu(A_n)\to 0$ as $n\to \infty$.
\end{enumerate}
In case (I), a desired sequence is obtained along the proof of \cite[Lemma 2.3]{js}.

Suppose that case (II) holds.
We apply Chifan-Ioana's argument in the proof of \cite[Lemma 10]{ci} that relies on the proof of \cite[Lemma 14]{ajzn}.
Put $T=T_2$.
Lemma \ref{lem-mix} implies that for any $\varepsilon >0$ and any measurable subsets $A$, $B$ of $\T$, there exists $l\in \N$ with
\[\mu(T^{-l}(A)\cup B)>(1-\varepsilon)-(1-\mu(A))(1-\mu(B)).\]
For any $\varepsilon >0$ and any measurable subset $A$ of $\T$ with $\mu(A)<1$, using this inequality inductively, we can find positive integers $l_i$ for each $i\in \N$ such that for any $m\in \N$, we have
\[\mu \left( \bigcup_{i=1}^mT^{-l_i}(A)\right)>1-\varepsilon -(1-\mu(A))^m.\]

For each $n\in \N$, let $m_n$ be the largest integer that is not larger than $(2\mu(A_n))^{-1}$.
We apply the argument in the last paragraph to $A_n$, and then find positive integers $l_1,\ldots, l_{m_n}$ such that $\mu(B_n)>1-n^{-1}-(1-\mu(A_n))^{m_n}$, where we set
\[B_n=\bigcup_{i=1}^{m_n}T^{-l_i}(A_n).\]
Since $\mu(A_n)\to 0$ as $n\to \infty$, for any sufficiently large $n\in \N$, we have $1-100^{-1}-e^{-1/2}\leq \mu(B_n)\leq m_n\mu(A_n)\leq 1/2$.
On the other hand, for any $k=p, q$, the inequality
\begin{align*}
\mu(T_k^{-1}(B_n)\bigtriangleup B_n)&\leq \sum_{i=1}^{m_n}\mu(T_k^{-1}(T^{-l_i}(A_n))\bigtriangleup T^{-l_i}(A_n))=m_n\mu(T_k^{-1}(A_n)\bigtriangleup A_n)\\
&\leq \frac{\mu(T_k^{-1}(A_n)\bigtriangleup A_n)}{2\mu(A_n)}
\end{align*}
holds.
The proof in case (II) is therefore reduced to that of case (I).
\end{proof}

Let $(B_n)_{n\in \N }$ be the sequence of measurable subsets of $\T$ in Lemma \ref{lem-a}.
For each $n\in \N$, define a function $w_n$ in $L^{\infty}(\T)$ by $w_n(z)=1$ for $z\in B_n$, and $w_n(z)=-1$ for $z\in \T \setminus B_n$.
The function $w_n$ is then a unitary element in $L\langle a \rangle$ with $\tau(w_n)=0$, and we have
\begin{align*}
\Vert u_tw_nu_t^*-w_n\Vert_2&\leq \Vert w_n-T_p^{-1}w_n\Vert_2+\Vert u_t(T_p^{-1}w_n)u_t^*-w_n\Vert_2\\
&=\Vert w_n-T_p^{-1}w_n\Vert_2+\Vert T_q^{-1}w_n-w_n\Vert_2\\
&=2\mu(T_p^{-1}(B_n)\bigtriangleup B_n)^{1/2}+2\mu(T_q^{-1}(B_n)\bigtriangleup B_n)^{1/2}.
\end{align*}
It follows that $\Vert u_tw_nu_t^*-w_n\Vert_2\to 0$ as $n\to 0$.
For any $\gamma \in E$ and any $n\in \N$, the equation $u_{\gamma} w_nu_{\gamma}^*=w_n$ holds because $a$ is a central element of $E$.
It turns out that $L\Gamma$ has property Gamma.

The proof of Theorem \ref{thm-gamma} is completed.


\end{document}